\documentclass[12pt,a4paper]{article}  
\usepackage{t1enc}
\usepackage[latin1]{inputenc}
\usepackage[english]{babel}
\usepackage{graphicx}
\usepackage{amsmath}
\usepackage{amssymb}
\usepackage{amsthm}
\font\bBB=msbm10

\def\bBR{\mbox{\bBB R}}

\def\bBZ{\mbox{\bBB Z}}

\def\R3{\bBR ^3}

\def\Z{\bBZ}
\def\Zz{\bBZ_2}

\begin{document}
\title {A Class of Efficient Presentations of Finite Simple Groups}
\author{
   Orlin Stoytchev\thanks{American University in Bulgaria, 2700 Blagoevgrad, Bulgaria,  ostoytchev@aubg.edu} }
\date{}
\maketitle

\abstract {We exhibit a new presentation of the (equilateral) Von Dyck groups  $D(2,3,n), \ n\ge 3$, in terms of two generators of order $n$ satisfying three relations, one of which is Artin's braid relation. By dropping the relation which fixes the order of the generators we obtain the universal covering groups of the corresponding Von Dyck groups. In the cases $n=3,\, 4,\,5$, these are respectively the double covers of the finite rotational tetrahedral, octahedral and icosahedral groups. When $n\ge 6$ we obtain infinite covers of the corresponding infinite Von Dyck groups. The interesting cases arise for $n\ge 7$ when these groups act as discrete groups of isometries of the hyperbolic plane. Imposing a suitable third relation we obtain a host of (efficient) presentations of finite simple Chevalley groups of type $A_1$ as well as the sporadic Janko group $J_2$.   }

\section{Rotations, Braid Groups, Von Dyck Groups and Their Covers }

The Von Dyck groups $D(l,m,n)$, which many authors call rotational triangle groups, or just triangle groups, are the finitely presented groups with two generators and three relations as follows \cite{Cox}:
\begin{equation}\label{pres}
D(l,m,n):=\left<x,y\, |\, x^l=y^m=(xy)^n=\mathbf 1\right> \ ,
\end{equation}
where $l$, $m$, and $n$ are integers greater than 1. These groups have a geometric realization as discrete subgroups of the groups of isometries of a simply connected Riemann surface with constant positive curvature (sphere), zero curvature (Euclidean plane), or negative curvature (hyperbolic plane). They are generated by rotations about two of the vertices of a geodesic triangle with angles $\pi/l$, $\pi/m$, and $\pi/n$. The rotations are at angles twice the angle at each vertex. Obviously depending on whether the sum $\frac{1}{l} +\frac{1}{m} + \frac{1}{n}$ is greater than, equal, or less than 1, we get a geodesic triangle on the sphere, the Euclidean plane, or the hyperbolic plane, respectively. Further, since the sphere is compact, the requirement for discreteness imposes in this case a restriction on the size of the defining triangle. Figure \ref{VonDyck} illustrates the action of the generators $x$ and $y$ and explains the origin of the third relation in Equation \ref{pres}.\par
\begin{figure}
\centering
\includegraphics[width=150mm]{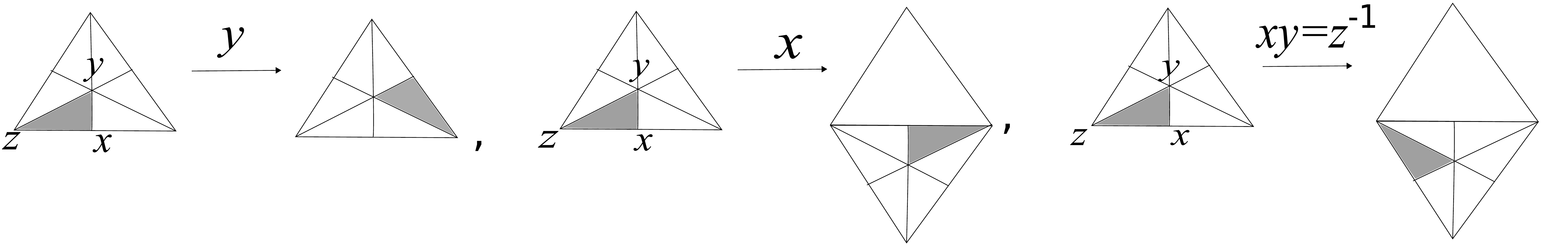}
\caption{The standard generators of $D(2,3,n)$ }
\label{VonDyck}
\end{figure}
The action of the Von Dyck group defined in this way covers the surface by copies of the defining triangle and its mirror image across one of the sides and produces a tessellation of the surface. The defining triangle and its reflection constitute a fundamental domain for the action of the group. It is quite obvious from the geometric picture that the Von Dyck groups are finite in the spherical case and infinite otherwise.\par
We concentrate further on the most symmetric cases $D(2,3,n)$, $n\ge 3$, which we call equilateral, as the large triangle obtained by the action of $y$ on the fundamental domain is equilateral. Our picture (Figure \ref{VonDyck}) corresponds to such a case. The sum $\frac{1}{2} +\frac{1}{3} + \frac{1}{n}$ is greater than 1 only when $n=3,4,5$. The sphere is tessellated by the equilateral triangles described above and $n$ is the number of triangles that meet at each vertex. Thus, we get a spherical model of the tetrahedron when $n=3$,  of the octahedron when $n=4$, and of the icosahedron when $n=5$. The groups of rotational symmetries of these regular polyhedra are precisely the finite groups $D(2,3,3)$ of order 12, $D(2,3,4)$ of order 24 and $D(2,3,5)$ of order 60. The hyperbolic case is of greater interest as many interesting finite groups are obtained as factors of Von Dyck groups.
\par
The geometric picture suggests a different set of generators for $D(2,3,n)$. We consider the equilateral geodesic triangle and two rotations $a$ and $b$ about two of the vertices (which we also denote by $a$ and $b$) at an angle $2\pi/n$, which is the angle at each vertex. It is immediate that $x=aba$ and $y=ab$. However, as shown on Figure \ref{Artin}, we also have $x=bab$, which implies that the two generators satisfy Artin's braid relation $aba=bab$.\par
Further, $a$ and $b$ satisfy $ab^2a=b^{n-2}$, which is illustrated on Figure \ref{Rel2}. Notice that this relation carries information about the number $n$ of equilateral triangles that meet at each vertex, when the surface is tessellated by them.
\begin{figure}
\centering
\includegraphics[width=100mm]{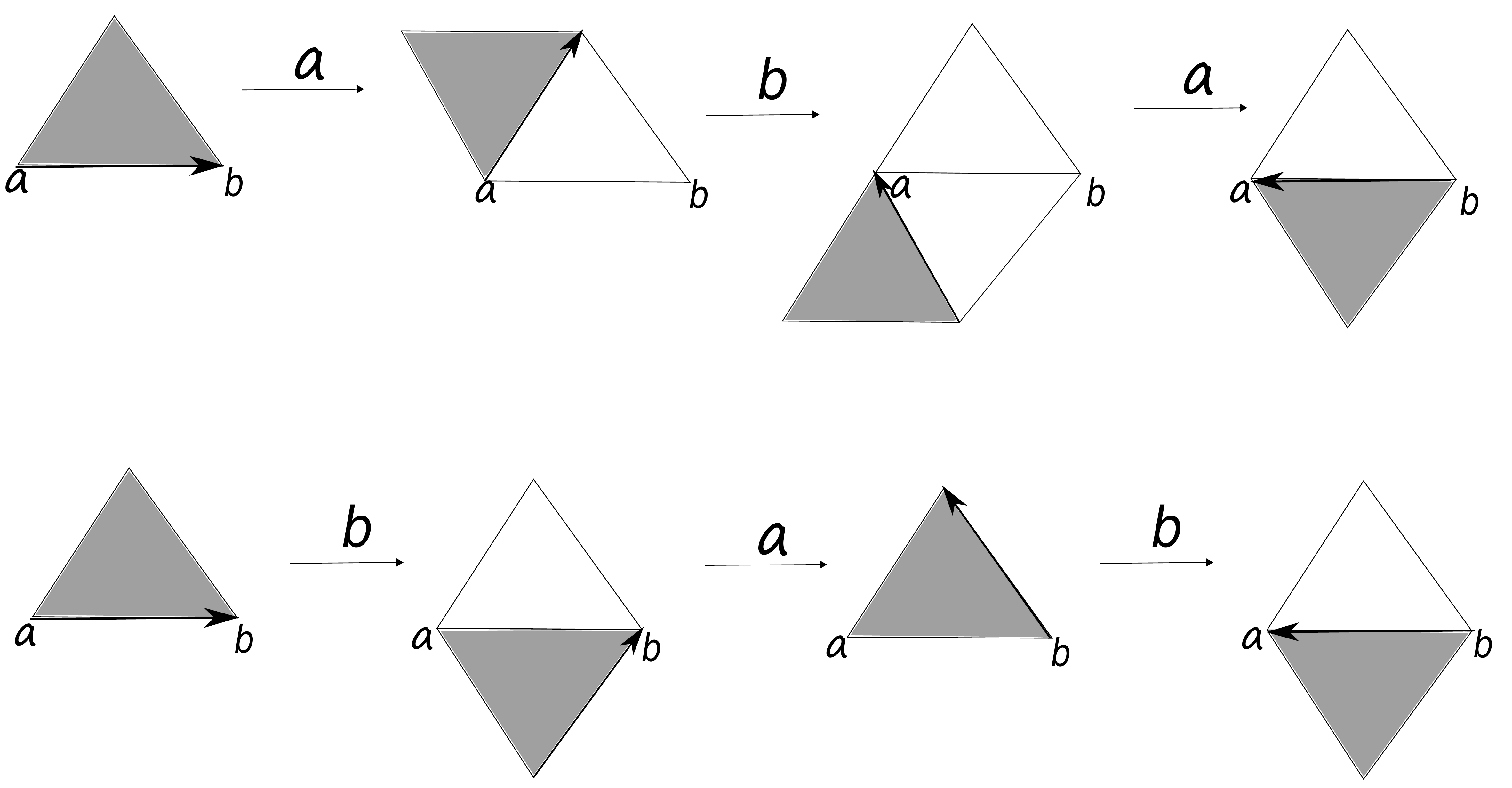}
\caption{An illustration of Artin's relation $aba=bab$}
\label{Artin}
\end{figure}

\begin{figure}
\centering
\includegraphics[width=130mm]{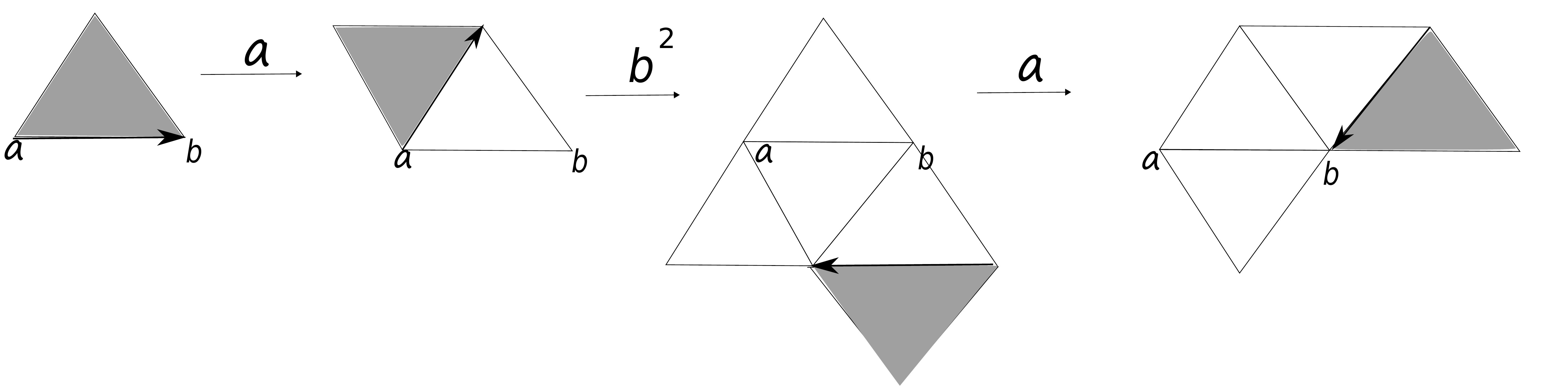}
\caption{The relation $ab^2a=b^{n-2}$}
\label{Rel2}
\end{figure}
The third relation is a condition on the order of (both) generators, namely $a^n=\mathbf 1$. 
\newtheorem{prop}{Proposition}
\begin{prop}\label{Prop1}
The group $D_n
:=\left< a,b\  |\ aba=bab,\,ab^2a=b^{n-2},\,a^n=\mathbf 1 \right> $ is isomorphic to the Von Dyck group $D(2,3,n)$.
\end{prop}
\begin{proof}
We notice that the group $D_n$ is a factor group of the braid group with three strands $B_3$. The latter has an obvious automorphism generated by the the map $a\leftrightarrow b$, so any additional relation for $a$ and $b$ automatically implies a symmetric relation with $a$ and $b$ interchanged. This observation shows why having two generators that satisfy Artin's braid relation is economical. In particular it follows that $b^n=\mathbf 1$ and $ba^2b=a^{n-2}$. We can also see these directly. Indeed, applying Artin's braid relation $n$ times we get
$$
bab^na^{-1}b^{-1}=babb^{n-1}a^{-1}b^{-1}=abab^{n-1}a^{-1}b^{-1}=...=a^n baa^{-1}b^{-1}=a^n=\mathbf 1\ ,
$$
which obviously implies $b^n=\mathbf 1$. In a similar way we write
$$bab^{n-2}a^{-1}b^{-1}=a^{n-2}\ .
$$
Now we use the second relation for $D_n$ inside the expression on the left:
$$a^{n-2}=baab^2aa^{-1}b^{-1}=ba^2b\ .
$$
Assume now that the group $D_n$ is defined as above and set $x:=aba$ and $y:=ab$. Then
\begin{eqnarray}\nonumber
x^2&=&abaaba=aa^{n-2}a=a^n=\mathbf 1\ ,\\ \nonumber
y^3&=&ababab=abaaba=x^2=\mathbf 1\ ,\\ \nonumber
(xy)^n&=&(abaab)^n=(a^{n-1})^n=(a^{-1})^n=\mathbf 1 \ .\\ \nonumber
\end{eqnarray}
Conversely, let $x$ and $y$ be the standard generators of $D(2,3,n)$ as in Equation \ref{VonDyck} and set $a=y^2x$ and $b=xy^2$. Then $aba=y^2xxy^2y^2x=y^6x=x$. At the same time $bab=xy^2y^2xxy^2=xy^6=x$, so Artin's braid relation holds for $a$ and $b$. Next, $a^n=(y^2x)^n=(y^{-1}x^{-1})^n=(xy)^{-n}=\mathbf 1$ and, as we saw, this implies also $b^n=\mathbf 1$. Finally, $ab^2a=y^4xy^4x=(yx)^2=(y^{-2}x^{-1})^2=(xy^2)^{-2}=b^{-2}=b^{n-2}$.
\end{proof}
We now considers the group $\overline D_n:=\left< a,b\  |\ aba=bab,\,ab^2a=b^{n-2} \right> $. In other words we drop the condition on the order of $a$ and $b$. We still have the action of $\overline D_n$ on the tessellated sphere, Euclidean plane, or hyperbolic plane, except that now the action is not faithful. Further, we can think of continuous motions of the respective surfaces, so that $a$ can be considered as the continuous rotation about the vertex $a$ from angle $0$ to angle $2\pi/n$, $b$ is a similar continuous rotation about the vertex $b$, $ab$ means first performing the motion given by $b$ and then the one given by $a$, and so on. We can now state that the two relations in the definition of $\overline D_n$ hold in a stronger, homotopy sense. Namely, the motion determined by $aba$ can be continuously deformed to the one determined by $bab$ and the motion determined by $ab^2a$ can be continuously deformed to the one determined by $b^{n-2}$. (Notice that the full twists determined by $a^n$ or $b^n$ cannot be deformed to the trivial motion. Notice also that the spherical case is substantially different from the other cases, as a rotation by $4\pi$ of the sphere is continuously deformable to the trivial motion, which is not true for the other cases.)
\begin{prop}\label{Prop2}
The group $\overline D_n$ is the universal central extensions of $D(2,3,n)$. When $n=3,\,4, \,5$ the extension is by $\Zz$ and the group obtained is the binary tetrahedral, binary octahedral, and binary icosahedral group, respectively. We have the short exact sequence
\begin{equation}
1\longrightarrow \Zz\longrightarrow \overline D_n\longrightarrow D(2,3,n)\longrightarrow 1\ .
\label{double}
\end{equation}
In the case $n\ge 6$ we get  central extensions by $\Z$ of the infinite Von Dyck group $D(2,3,n)$:
\begin{equation}
1\longrightarrow \Z\longrightarrow \overline D_n\longrightarrow D(2,3,n)\longrightarrow 1\ .
\label{universal}
\end{equation}
\end{prop}
\begin{proof}
We will use a topological argument. First notice that $a^n=b^n$ is central in $\overline D_n$. Indeed, using the two relations for $\overline D_n$ we see that
$$
ababab=abaaba=a^n\ ,
$$
and similarly $b^n=bababa=ababab=a^n$. But the element $(ab)^3$ is central in the braid group $B_3$ and therefore also in $\overline D_n$, which can be seen immediately by applying Artin's relation. In the spherical case the group $\overline D_n$ has the meaning of the group of homotopy classes of paths in $SO(3)$, starting at the identity and ending at some element of the finite tetrahedral, octahedral, or icosahedral subgroup of $SO(3)$. The projection $\overline D_n \rightarrow D(2,3,n)$ obtained by imposing the additional relation $a^n=\mathbf 1$ is a restriction of the double cover $SU(2) \rightarrow SO(3)$ and is therefore a double cover itself. In particular $a^n$ has the meaning of a rotation from $0$ to $2\pi$ about some axis, which is not topologically trivial, but, as is well known, $(a^n)^2$ is. \par
The Euclidean case ($n=6$) and the hyperbolic case ($n>6$) are essentially the same. The group generated by rotations about two different points is a subgroup of the orientation preserving isometries. In the Euclidean case this is the group $E^+(2)=O(2)\ltimes T(2)$, where $T(2)$ denotes the group of translations of $\bBR ^2$. In the hyperbolic case the group is $PSL(2,\bBR )$. In both cases the topology of these groups is that of $S^1\times \bBR^2$ and the universal covering space is topologically $\bBR\times \bBR ^2$. We have a short exact sequence 
$$
1\longrightarrow \Z\longrightarrow \overline {PSL(2,\bBR)}\longrightarrow PSL(2,\bBR)\longrightarrow 1
$$
and a similar one for $E^+(2)$. In both cases the covering map $\overline D_n \rightarrow D(2,3,n)$ obtained by imposing the additional relation $a^n=\mathbf 1$ is a restriction of the universal covering maps just described. In particular $a^n$ has infinite order.
\end{proof}
In the proof of the last proposition we used the fact that the fundamental group of $SO(3)$ is $\Zz$ to ascertain that $\overline D_n$ is a double cover of $D(2,3,n)$ when $n=3,4,5$. This connection between $\pi_1(SO(3))$ and the groups $\overline D_n$ can be used in the opposite direction. For any one of the three subgroups of $SO(3)$ -- rotational tetrahedral, octahedral, or icosahedral group -- one can study the group $\overline D_n$, which is a group of homotopy classes of paths in $SO(3)$ starting at the identity and ending at an element of the respective polyhedral group. Then one shows that the central element $a^n=b^n$, which is the class of closed paths in $SO(3)$,  has order two and $\overline D_n$ is indeed a double (Schur) cover of the respective rotational polyhedral group. This approach has been pursued in detail in \cite{Ina} for the octahedral case, where the relations for $\overline D_4$ are derived by considering contractible closed paths. In fact the scheme works for $SO(m)$ with any $m\ge 3$. In all cases we get a group $\overline D_4$ which is a factor group of the braid group with $m$ strands $B_m$ obtained by imposing the single additional relation $ab^2a=b^2$ for the first two generators.\par
It may be instructive to see why the relations for $\overline D_n$ imply that $a^n$ has order two when $n=3,4,5$. Using $ab^2=b^{n-2}a^{-1}$ multiple times we can write
$$
a^nb^2=a^{n-1}b^{n-2}a^{-1}=a^{n-1}b^2aa^{-1}b^{n-4}a^{-1}=...=b^{n-2}a^{-1}b^{n-4}a^{-1}b^{n-4}\cdots a^{-1} \ .
$$
As $a^n$ is central, this implies
$$
a^n=(b^{n-4}a^{-1})^n\ .
$$
For $n=3$ this reads $a^3=(b^{-1}a^{-1})^3=(ab)^{-3}$. But we also know that $a^3=(ab)^3$, so $a^6=\mathbf 1$. For $n=4$ the formula implies $a^4=(a^{-1})^4$ and obviously $a^8=\mathbf 1$. For $n=5$ we have $a^5=(ba^{-1})^5$. Because of the symmetry between $a$ and $b$ we also have $b^5=(ab^{-1})^5$ and therefore $a^{10}=a^5b^5=\mathbf 1$.
\par
\section{Finite Simple Factors}
The smallest finite simple nonabelian group is the alternating group $A_5$ of order 60. It is isomorphic to the rotational icosahedral group, the Von Dyck group $D(2,3,5)$ and to  the Chevalley groups $A_1(4)$ and $A_1(5)$. Recall that the Chevalley groups are defined as $A_n(q):=PSL(n+1,q)$, i.e. the groups of $(n+1)\times (n+1)$ matrices over the finite field with $q$ elements, having unit determinant, modulo their center. The Chevalley groups form one of the infinite series of finite simple groups except the cases $A_1(2)$ and $A_1(3)$. 
All examples that we have found, except for the Janko group $J_2$, fall in the series $A_1(q)$ and we will adopt this notation henceforth. Our approach provides the following presentation:
$$
A_1(5)\cong \left< a,b\,|\,aba=bab, ab^2a=b^3, a^5=\mathbf 1\,\right>\ .
$$
This is, just as the standard one, a presentation with two generators and three relations and is therefore an {\it efficient presentation}. Recall that a presentation of a finite group is called {\it efficient} (or {\it balanced}) if the difference between the number of relations and the number of generators is equal to the rank of the Schur multiplier group. If a finite group has a nontrivial Schur multiplier it admits a central (Schur) extension, which is also finite. The minimal possible number of relations for the Schur extension is equal to the number of generators for otherwise the group would be infinite. Then, if the Schur multiplier is nontrivial, one needs an additional relation for each of its generators to recover the original group. In our case this is the relation $a^5=\mathbf 1$.  The Schur multiplier of $D(2,3,5)$ is $\Zz$ (rank 1) and therefore the presentation is efficient. The Schur extension is the double cover $\overline{A_1(5))}$ (the binary icosahedral group) with (efficient) presentation
$$
\overline{A_1(5)}\cong \left< a,b\,|\,aba=bab, ab^2a=b^3\,\right>\ .
$$
\par
The next finite simple group is the group $A_1(7)$ of order $168$. It is a Hurwitz group, i.e. a finite factor of the Von Dyck group $D(2,3,7)$. It is the group of automorphisms (orientation-preserving isometries) of a compact Riemann surface of genus 3 known as the Klein quartic. The fact that $A_1(7)$ is a factor of $D(2,3,7)$ has the following geometric interpretation: The hyperbolic plane is tessellated by equilateral geodesic triangles with angles $2\pi/7$. There is a discrete group of motions of the hyperbolic plane preserving this tessellation, which plays the role of a Fuchsian group determining the Klein quartic. Its automorphism group is the factor group of the normalizer in $PSL(2,\bBR)$ of this Fuchsian group by the Fuchsian group itself.\par
Hurwitz groups have been the subject of considerable interest (see, e.g., \cite{Conder}, \cite{Macbeth} for overviews of known facts) for many reasons. Among them is the result of Hurwitz that the biggest automorphism group of a compact Riemann surface of given genus is a Hurwitz group. Also, Hurwitz groups are quite abundant in the Chevalley series $A_1(q)$ and many other finite simple groups, including the Monster group, are known to be of this type. \par
Some properties of Hurwitz groups are shared by other groups $G$ which are finite factors of Von Dyck groups of type $D(2,3,n), \ n>6$. In particular we have the following generalization of a result described in \cite{Conder}: \\
Suppose that $\{2,3,n\}$ are pairwise coprime. Then\\
(a) The order of $G$ is divisible by $12n$.\\
(b) $G$ is perfect and it has a maximal normal subgroup $K$ such that $G/K$ is a nonabelian simple group.\par
We show (a) through a simple geometric argument. The Von Dyck group $D(2,3,n)$ acts on the hyperbolic plane preserving the tessellation by equilateral geodesic triangles with angles $2\pi/n$. The action is transitive on the set of triangles and further a given triangle can be mapped to any other in three different ways (cyclic permutations of the three ordered vertices). These properties are inherited by the compact Riemann surface when we factor by the action of a normal subgroup of  $D(2,3,n)$. In particular the tessellated surface becomes a generalized regular polyhedron whose faces are the regular triangles. Since there are $n$ triangles meeting at a vertex we have for the number of vertices $V$ and number of faces $F$ the formula $V=3F/n$ and  the relation between number of edges $E$ and number of faces is $E=3F/2$. We see that the Euler characteristic of the surface is $\chi=F(6-n)/2n$. Using the fact that $|G|=3F$ and expressing the Euler characteristic in terms of the genus we obtain
\begin{equation}
\label{genus}
g=1+\frac{|G|(n-6)}{12n}\ .
\end{equation}
Turning this formula around shows that $|G|$ is divisible by $12n$, since $n-6$ has no common factors with $12$ or with $n$.\par
The proof of (b) follows precisely the one in \cite{Conder}, namely it is almost immediate that the abelianization of $G$ is trivial and also any factor group of $G$ is also a factor group of $D(2,3,n)$ and therefore shares the same properties as $G$.\par
Our approach to finding efficient presentations to a class of finite simple groups (and their covers in some cases) is unsophisticated -- we consider the groups $\overline D_n:=\left< a,b\  |\ aba=bab,\,ab^2a=b^{n-2} \right> $ for different values of $n\ge 7$ and look for a single additional relation $W=W(a,b,a^{-1},b^{-1})=\mathbf 1$ that will produce a finite group. We must choose an element $W$ such that the normal closure of the subgroup generated by it acts properly discontinuously on the hyperbolic plane and is therefore a Fuchsian group. When $n=7$ we obtain in this way Hurwitz groups or their double cover. In the former case the order of the generators $a$ and $b$ turns out to be $7$, in the latter it is $14$. Of course not all Hurwitz groups obtained in this way are simple but we choose to omit the non-simple ones. A similar approach is used for other values of $n$. When $n$ is coprime with $2$ and $3$ the results mimic those for $n=7$; when $n$ is divisible by $2$ or $3$ the finite groups obtained are often cyclic.\par
The calculations were done using the computer system GAP -- Groups, Algorithms, Programming \cite{GAP} and essentially trying out different relations $W=\mathbf 1$ for which the GAP algorithm for enumerating the elements of $G$  would terminate in a reasonable amount of time and without memory overflow. If the order of $G$ coincided with that of a simple group, we used the``IsSimple'' routine to establish that $G$ is indeed the desired group. Unfortunately the GAP algorithm for checking simplicity is much less efficient than the one calculating the order of $G$. For small $n$ but large $|G|$ it took between several seconds and several weeks on a workstation with CPUs at 3.28 GHz for a calculation to finish. For example it took between one and two weeks to check simplicity of the group of order 604800 (the Janko group $J_2$). In some cases the calculations did not conclude by the time of the writing of this paper. For larger $n$ GAP would return an error message "the coset enumeration has defined more than ..... cosets" and all attempts to continue with a larger limit or without a limit did not work, possibly due to a bug in the algorithm.\par
In the following section we list the presentations that we have found. The ones marked with an asterisk are those for which simplicity has not been confirmed due to the limitations mentioned above. We have checked that all the groups are perfect. Still, the identification with the respective simple group should be treated as (very plausible) conjectures. (It may be mentioned that in all cases where the group was shown to be perfect and then the calculation for simplicity came to completion GAP gave response ``true''.)\par
The cases in which $G=\left< a,b\  |\ aba=bab,\,ab^2a=b^{n-2},\,W=\mathbf 1 \right>$ is a Schur cover of a simple group or is a simple group with trivial Schur multiplier this is not an efficient presentation and we have tried to find a presentation of the form $G=\left< a,b\  |\ aba=bab,\,ab^2a=W^{- 1}b^{n-2}\, \right>$. There are two groups for which the presentations found are not efficient. They are marked by a dagger.\par
We observe that a group $A_1(q)$ may occur as a factor group of $D(2,3,n)$ for several different values of $n$. Since the normal closure in $D(2,3,n)$ of the subgroup generated by $W$ is a Fuchsian group, the group $A_1(q)$ appears as an automorphism group of the corresponding compact Riemann surface and thus we may have several different surfaces with the same automorphism group. The genera of the surfaces are calculated according to Equation \ref{genus}.\par
Finally, it is interesting that for $p$ prime we have been able to find a presentation of $A_1(p)$ as a factor of $D(2,3,p)$ for all $7\le p \le 47$ and this is always the smallest factor of $D(2,3,p)$. It is natural to ask if this is true for any prime $p\ge 7$.
\par
\section {Results}
\begin{scriptsize}

\begin{center} \underline{$\mathbf{n=7}$ (Hurwitz groups)} \end{center}\par
$\bullet$ {\bf order 168} 
\begin{eqnarray}
A_1(7)&\cong& \left< a,b\  |\ aba=bab,\,ab^2a=b^5,\,(ab^{-1})^4=\mathbf 1 \right>\nonumber\\ 
& \cong &\left< a,b\  |\ aba=bab,\,ab^2a=b^5,\,(a^4b^{-3})^2=\mathbf 1 \right>\nonumber\\
&\cong& \left< a,b\  |\ aba=bab,\,ab^2a=(a^2b^{-2})^{\pm 3}b^5,\,a^7=\mathbf 1 \right>\nonumber
\end{eqnarray}
\par\bigskip
$\bullet$ {\bf order 336} \quad(Not simple - Schur cover)
\begin{eqnarray}
\overline{A_1(7)}&\cong& \left< a,b\  |\ aba=bab,\,ab^2a=(a^2b^{-2})^{\pm 3}b^5\, \right>\nonumber
\end{eqnarray}
\par\bigskip
$\bullet$ {\bf order 504} 
\begin{eqnarray}
{A_1(8)}&\cong& \left< a,b\  |\ aba=bab,\,ab^2a=(a^2b^{-2}ab^{-2})^{2}b^5\, \right>\nonumber\\
&\cong& \left< a,b\  |\ aba=bab,\,ab^2a=((ab^{-})^3 ab^{-2})^{2}b^5\, \right>\nonumber
\end{eqnarray}
\par\bigskip
$\bullet$ {\bf order 1092} 
\begin{eqnarray}
{A_1(13)}&\cong& \left< a,b\  |\ aba=bab,\,ab^2a=b^5,\, (ab^{-1})^6=\mathbf 1\,\right>\nonumber\\
&\cong& \left< a,b\  |\ aba=bab,\,ab^2a=b^5,\, (a^3b^{-4})^3=\mathbf 1\,\right>\nonumber\\
&\cong& \left< a,b\  |\ aba=bab,\,ab^2a=b^5,\, (ab^{-1}a^2b^{-2}a^3b^{-3})^2=\mathbf 1\,\right>\nonumber\\
&\cong& \left< a,b\  |\ aba=bab,\,ab^2a=(a^3b^{-3})^{\pm 3}b^5,\, a^7=\mathbf 1\,\right>\nonumber
\end{eqnarray}
\par\bigskip
$\bullet$ {\bf order 2184} \quad(Not simple - Schur cover)
\begin{eqnarray}
\overline{A_1(13)}&\cong& \left< a,b\  |\ aba=bab,\,ab^2a=(a^3b^{-3})^{\pm 3}b^5\,\right>\nonumber
\end{eqnarray}
\par\bigskip
$\bullet$ {\bf order 9828} 
\begin{eqnarray}
{A_1(27)}&\cong& \left< a,b\  |\ aba=bab,\,ab^2a=b^5,\, ((ab^{-1})^5b^{-1})^3=\mathbf 1\, \right>\nonumber\\
&\cong& \left< a,b\  |\ aba=bab,\,ab^2a=b^5,\, ((ab^{-1})^3b^{-1}a^3b^{-2})^3=\mathbf 1\, \right>\nonumber\\
&\cong& \left< a,b\  |\ aba=bab,\,ab^2a=b^5,\, ((ab^{-1})^3b^{-1}(ab^{-1})^2a)^2=\mathbf 1\, \right>\nonumber
\end{eqnarray}
 \par\bigskip
$\bullet$ {\bf order 12180} 
\begin{eqnarray}
{A_1(29)}&\cong& \left< a,b\  |\ aba=bab,\,ab^2a=b^5,\, ((ab^{-1})^4b^{-1})^3=\mathbf 1\, \right>\nonumber\\
&\cong& \left< a,b\  |\ aba=bab,\,ab^2a=b^5,\, (ab^{-1}a^2b^{-2}a^3b^{-2}a^2b^{-1})^2=\mathbf 1\, \right>\nonumber\\
&\cong& \left< a,b\  |\ aba=bab,\,ab^2a=(a^2b^{-2})^{\pm 5}b^5,\, a^7=\mathbf 1\, \right>\nonumber\\
&\cong& \left< a,b\  |\ aba=bab,\,ab^2a=(a^3b^{-3})^{\pm 5}b^5,\, a^7=\mathbf 1\, \right>\nonumber\\
&\cong& \left< a,b\  |\ aba=bab,\,ab^2a=(a^4b^{-4})^{\pm 5}b^5,\, a^7=\mathbf 1\, \right>\nonumber
\end{eqnarray}
\par\bigskip
$\bullet$ {\bf order 24360} \quad(Not simple - Schur cover)
\begin{eqnarray}
\overline{A_1(29)}&\cong& \left< a,b\  |\ aba=bab,\,ab^2a=(a^2b^{-2})^{\pm 5}b^5\, \right>\nonumber\\
&\cong& \left< a,b\  |\ aba=bab,\,ab^2a=(a^3b^{-3})^{\pm 5}b^5\, \right>\nonumber\\
&\cong& \left< a,b\  |\ aba=bab,\,ab^2a=(a^4b^{-4})^{ 5}b^5\, \right>\nonumber
\end{eqnarray}
\par\bigskip
$\bullet$ {\bf order 34440} 
\begin{eqnarray}
{A_1(41)}&\cong& \left< a,b\  |\ aba=bab,\,ab^2a=b^5,\, (a^2b^{-1}ab^{-2})^4=\mathbf 1\, \right>\nonumber\\
&\cong& \left< a,b\  |\ aba=bab,\,ab^2a=b^5,\, (a^2b^{-2}ab^{-3})^4=\mathbf 1\, \right>\nonumber\\
&\cong& \left< a,b\  |\ aba=bab,\,ab^2a=b^5,\, (ab^{-1}ab^{-2}ab^{-3})^4=\mathbf 1\, \right>\nonumber
\end{eqnarray}
\par\bigskip
$\bullet$ {\bf order 39732} 
\begin{eqnarray}
{A_1(43)}&\cong& \left< a,b\  |\ aba=bab,\,ab^2a=b^5,\, (ab^{-1}a^2b^{-2}a^3b^{-3})^3=\mathbf 1\, \right>\nonumber\\
&\cong& \left< a,b\  |\ aba=bab,\,ab^2a=b^5,\, ((ab^{-2})^4a^2b^{-1}ab^{-1})^2=\mathbf 1\, \right>\nonumber\\
&\cong& \left< a,b\  |\ aba=bab,\,ab^2a=(a^2b^{-1}ab^{-2}ab^{-1})^{\pm3}b^5,\, a^7=\mathbf 1\, \right>\nonumber
\end{eqnarray}
\par\bigskip
$\bullet$ {\bf order 79464} \quad(Not simple - Schur cover)
\begin{eqnarray}
\overline{A_1(43)}&\cong& \left< a,b\  |\ aba=bab,\,ab^2a=(a^2b^{-1}ab^{-2}ab^{-1})^{\pm3}b^5\, \right>\nonumber
\end{eqnarray}
\par\bigskip
$\bullet$ {\bf order 178920} 
\begin{eqnarray}
{A_1(71)}&\cong& \left< a,b\  |\ aba=bab,\,ab^2a=b^5,\, ((ab^{-1})^3b^{-1})^4=\mathbf 1\, \right>\nonumber\\
&\cong& \left< a,b\  |\ aba=bab,\,ab^2a=b^5,\, ((a^2b^{-2})^2b^{-2})^4=\mathbf 1\, \right>\nonumber\\
&\cong& \left< a,b\  |\ aba=bab,\,ab^2a=b^5,\, ((ab^{-1})^5a^4b^{-3})^2=\mathbf 1\, \right>\nonumber\\
&\cong& \left< a,b\  |\ aba=bab,\,ab^2a=b^5,\, ((a^3b^{-3})^2(ab^{-1})^2b^{-2})^2=\mathbf 1\, \right>\nonumber
\end{eqnarray}
\par\bigskip
$\bullet$ {\bf order 285852} 
\begin{eqnarray}
{A_1(83)}&\cong& \left< a,b\  |\ aba=bab,\,ab^2a=b^5,\, ((ab^{-1})^5b^{-1}(ab^{-1})^2a)^2=\mathbf 1\, \right>\quad ^*\nonumber\\
&\cong& \left< a,b\  |\ aba=bab,\,ab^2a=b^5,\, ((ab^{-3})^2(a^2b^{-1})^2(ab^{-1})^2b^{-1}a)^2=\mathbf 1\, \right>\quad ^*\nonumber\\
&\cong& \left< a,b\  |\ aba=bab,\,ab^2a=b^5,\, ((a^2b^{-2})^3b^{-1}a^2b^{-3})^2=\mathbf 1\, \right>\quad ^*\nonumber\\
&\cong& \left< a,b\  |\ aba=bab,\,ab^2a=b^5,\, ((a^3b^{-1})^3(a^2b^{-2})^3(ab^{-1})^3)^2=\mathbf 1\, \right>\quad ^* \nonumber
\end{eqnarray}
\par\bigskip
$\bullet$ {\bf order 456288} 
\begin{eqnarray}
{A_1(97)}&\cong& \left< a,b\  |\ aba=bab,\,ab^2a=b^5,\, (a^2b^{-2}ab^{-2}(ab^{-1})^2a^2b^{-3}ab^{-1})^2=\mathbf 1\, \right>\quad ^*\nonumber\\
&\cong& \left< a,b\  |\ aba=bab,\,ab^2a=b^5,\, ((ab^{-2})^2(a^3b^{-1})^3(a^2b^{-2})^2)^2=\mathbf 1\, \right>\quad ^*\nonumber\\
&\cong& \left< a,b\  |\ aba=bab,\,ab^2a=b^5,\, (a^2b^{-2}(ab^{-1})^2(a^2b^{-3})^2a^3b^{-1})^2=\mathbf 1\, \right>\quad ^*\nonumber
\end{eqnarray}
\par\bigskip
$\bullet$ {\bf order 604800} 
\begin{eqnarray}
{J_2}&\cong& \left< a,b\  |\ aba=bab,\,ab^2a=b^5,\, (ab^{-1}ab^{-3}a^3b^{-1})^{\pm3}a^7=\mathbf 1\, \right>\nonumber\\
&\cong& \left< a,b\  |\ aba=bab,\,ab^2a=(ab^{-1}ab^{-3}a^3b^{-1})^{\pm3}b^5,\, a^7=\mathbf 1\, \right>\nonumber
\end{eqnarray}
\par\bigskip
$\bullet$ {\bf order 1209600} \quad(Not simple - Schur cover)
\begin{eqnarray}
\overline{J_2}
&\cong& \left< a,b\  |\ aba=bab,\,ab^2a=(ab^{-1}ab^{-3}a^3b^{-1})^3b^5\, \right>\nonumber
\end{eqnarray}
\par\bigskip
$\bullet$ {\bf order 721392} 
\begin{eqnarray}
{A_1(113)}&\cong& \left< a,b\  |\ aba=bab,\,ab^2a=b^5,\, ((ab^{-1})^3b^{-1}(a^2b^{-1})^4b^{-1})^2=\mathbf 1\, \right>\quad ^*\nonumber\\
&\cong& \left< a,b\  |\ aba=bab,\,ab^2a=b^5,\, ((ab^{-1})^5b^{-1}(ab^{-1})^3b^{-1})^2=\mathbf 1\, \right>\quad ^*\nonumber\\
&\cong& \left< a,b\  |\ aba=bab,\,ab^2a=b^5,\, ((a^3b^{-2})^2(a^2b^{-2})^2(ab^{-1})^3)^2=\mathbf 1\, \right>\quad ^*\nonumber
\end{eqnarray}
\par\bigskip
$\bullet$ {\bf order 976500} 
\begin{eqnarray}
{A_1(125)}&\cong& \left< a,b\  |\ aba=bab,\,ab^2a=b^5,\, ((ab^{-3})^4(a^2b^{-2})^2 (a^3b^{-1})^4)^2=\mathbf 1\, \right>\quad ^*\nonumber
\end{eqnarray}
\par\bigskip
$\bullet$ {\bf order 1024128} 
\begin{eqnarray}
{A_1(127)}&\cong& \left< a,b\  |\ aba=bab,\,ab^2a=b^5,\, (ab^{-1}(a^2b^{-3})^3(a^3b^{-1})^2)^2=\mathbf 1\, \right>\nonumber
\end{eqnarray}
\par\bigskip
$\bullet$ {\bf order 1342740} 
\begin{eqnarray}
{A_1(139)}&\cong& \left< a,b\  |\ aba=bab,\,ab^2a=b^5,\, ((a^2b^{-2})^6b)^2=\mathbf 1\, \right>\quad ^*\nonumber\\
&\cong& \left< a,b\  |\ aba=bab,\,ab^2a=b^5,\, (ab^{-2}(a^2b^{-2})^5)^2=\mathbf 1\, \right>\quad ^*\nonumber\\
&\cong& \left< a,b\  |\ aba=bab,\,ab^2a=b^5,\, (a^2b^{-2}a(ab^{-1})^6b^{-2})^2=\mathbf 1\, \right>\quad ^*\nonumber
\end{eqnarray}
\par\bigskip
$\bullet$ {\bf order 2328648} 
\begin{eqnarray}
{A_1(167)}&\cong& \left< a,b\  |\ aba=bab,\,ab^2a=b^5,\, ((a^3b^{-1})^3(ab^{-1})^7)^2=\mathbf 1\, \right>\quad ^*\nonumber\\
&\cong& \left< a,b\  |\ aba=bab,\,ab^2a=b^5,\, ((a^2b^{-2})^2ab^{-1}(a^2b^{-3})^2(a^3b^{-1})^2)^2=\mathbf 1\, \right>\quad ^*\nonumber
\end{eqnarray}
\par\bigskip
$\bullet$ {\bf order 5544672} 
\begin{eqnarray}
{A_1(223)}
&\cong& \left< a,b\  |\ aba=bab,\,ab^2a=b^5,\, ((a^2b^{-2})^3(a^2b^{-3})^2(a^3b^{-1})^2)^2=\mathbf 1\, \right>\quad ^*\nonumber
\end{eqnarray}
\par\bigskip
$\bullet$ {\bf order 7906500} 
\begin{eqnarray}
{A_1(251)}
&\cong& \left< a,b\  |\ aba=bab,\,ab^2a=b^5,\, ((ab^{-1})^3ab^{-2}(a^3b^{-2})^3)^2=\mathbf 1\, \right>\quad ^*\nonumber
\end{eqnarray}
\par\bigskip

\begin{center} \underline{$\mathbf{n=8}$} \end{center}\par
$\bullet$ {\bf order 2448} 
\begin{eqnarray}
{A_1(17)}
&\cong& \left< a,b\  |\ aba=bab,\,ab^2a=(a^2b^{-3})^{-3}b^6,\, a^8=\mathbf 1\, \right>\nonumber
\end{eqnarray}
\par\bigskip
$\bullet$ {\bf order 4896} \quad(Not simple - Schur cover)
\begin{eqnarray}
\overline{A_1(17)}
&\cong& \left< a,b\  |\ aba=bab,\,ab^2a=(a^2b^{-3})^{-3}b^6\, \right>\nonumber
\end{eqnarray}
\par\bigskip
$\bullet$ {\bf order 14880} 
\begin{eqnarray}
{A_1(31)}
&\cong& \left< a,b\  |\ aba=bab,\,ab^2a=(ab^{-2})^{-5}a^{-8}b^6,\, a^8=\mathbf 1\, \right>\nonumber
\end{eqnarray}
\par\bigskip
$\bullet$ {\bf order 29760} \quad(Not simple - Schur cover)
\begin{eqnarray}
\overline{A_1(31)}
&\cong& \left< a,b\  |\ aba=bab,\,ab^2a=(ab^{-2})^{-5}a^{-8}b^6\, \right>\nonumber
\end{eqnarray}
\par\bigskip
$\bullet$ {\bf order 51888} 
\begin{eqnarray}
{A_1(47)}
&\cong& \left< a,b\  |\ aba=bab,\,ab^2a=(ab^{-2}a^2b^{-2})^{-3}b^6,\, a^8=\mathbf 1\, \right>\nonumber
\end{eqnarray}
\par\bigskip
$\bullet$ {\bf order 721392} 
\begin{eqnarray}
{A_1(113)}
&\cong& \left< a,b\  |\ aba=bab,\,ab^2a=(a^2b^{-2}a^2b^{-3})^{\pm3}b^6,\, a^8=\mathbf 1\, \right>\quad ^*\nonumber
\end{eqnarray}
\par\bigskip

\begin{center} \underline{$\mathbf{n=9}$} \end{center}\par
$\bullet$ {\bf order 504} 
\begin{eqnarray}
{A_1(8)}
&\cong& \left< a,b\  |\ aba=bab,\,ab^2a=((ab^{-1})^3b^{-1})^2b^7\, \right>\nonumber\\
&\cong& \left< a,b\  |\ aba=bab,\,ab^2a=(a^3b^{-4})^2b^7\, \right>\nonumber\\
&\cong& \left< a,b\  |\ aba=bab,\,ab^2a=(ab^{-1}ab^{-3})^2b^7\, \right>\nonumber
\end{eqnarray}
 \par\bigskip
$\bullet$ {\bf order 2448} 
\begin{eqnarray}
{A_1(17)}
&\cong& \left< a,b\  |\ aba=bab,\,ab^2a=b^7,\, (ab^{-2})^4=\mathbf 1 \,\right>\nonumber\\
&\cong& \left< a,b\  |\ aba=bab,\,ab^2a=b^7,\, (ab^{-3})^4=\mathbf 1 \,\right>\nonumber
\end{eqnarray}
\par\bigskip
$\bullet$ {\bf order 3420} 
\begin{eqnarray}
{A_1(19)}
&\cong& \left< a,b\  |\ aba=bab,\,ab^2a=b^7,\, (a^3b^{-3}ab^{-3})^2=\mathbf 1 \,\right>\nonumber
\end{eqnarray}
\par\bigskip
$\bullet$ {\bf order 25308} 
\begin{eqnarray}
{A_1(37)}
&\cong& \left< a,b\  |\ aba=bab,\,ab^2a=b^7,\, (a^2b^{-2}a^2b^{-3})^2=\mathbf 1 \,\right>\nonumber\\
&\cong& \left< a,b\  |\ aba=bab,\,ab^2a=b^7,\, (ab^{-1}ab^{-2}a^2b^{-3})^2=\mathbf 1 \,\right>\nonumber
\end{eqnarray}
\par\bigskip
$\bullet$ {\bf order 178920} 
\begin{eqnarray}
{A_1(71)}
&\cong& \left< a,b\  |\ aba=bab,\,ab^2a=b^7,\, (a^2b^{-2}a^2b^{-2}a^2b^{-6})^2=\mathbf 1 \,\right>\quad ^*\nonumber
\end{eqnarray}
\par\bigskip

\begin{center} \underline{$\mathbf{n=10}$} \end{center}\par
$\bullet$ {\bf order 60} 
\begin{eqnarray}
{A_1(5)}
&\cong& \left< a,b\  |\ aba=bab,\,ab^2a=b^8,\, (ab^{-2})^3=\mathbf 1 \,\right>\nonumber\\&\cong& \left< a,b\  |\ aba=bab,\,ab^2a=(ab^{-2})^{3}b^8,\, a^5=\mathbf 1 \,\right>\nonumber
\end{eqnarray}
Note: The first presentation exhibits $A_1(5)$ as a factor of $D(2,3,10)$. Even so, the order of $a$ and $b$ turns out to be $5$, so the Riemann surface is the sphere triangulated by $20$ equilateral geodesic triangles (a spherical icosahedron).\par\bigskip
$\bullet$ {\bf order 120} \quad(Not simple - Schur cover)
\begin{eqnarray}
\overline{A_1(5)}
&\cong& \left< a,b\  |\ aba=bab,\,ab^2a=(ab^{-2})^{3}b^8=\mathbf 1 \,\right>\nonumber
\end{eqnarray}
\par\bigskip
$\bullet$ {\bf order 3420} 
\begin{eqnarray}
{A_1(19)}
&\cong& \left< a,b\  |\ aba=bab,\,ab^2a=(a^2b^{-3})^3b^8,\, a^{10}=\mathbf 1 \,\right>\nonumber
\end{eqnarray}
 \par\bigskip
$\bullet$ {\bf order 6840} \quad(Not simple - Schur cover)
\begin{eqnarray}
\overline{A_1(19)}
&\cong& \left< a,b\  |\ aba=bab,\,ab^2a=(a^2b^{-3})^3b^8 \,\right>\nonumber
\end{eqnarray}
 \par\bigskip
$\bullet$ {\bf order 34440} 
\begin{eqnarray}
{A_1(41)}
&\cong& \left< a,b\  |\ aba=bab,\,ab^2a=(ab^{-1}ab^{-2})^{\pm3}b^8,\, a^{10}=\mathbf 1 \,\right>\nonumber
\end{eqnarray}
\par\bigskip
$\bullet$ {\bf order 68880} \quad(Not simple - Schur cover)
\begin{eqnarray}
\overline{A_1(41)}
&\cong& \left< a,b\  |\ aba=bab,\,ab^2a=(ab^{-1}ab^{-2})^{\pm3}b^8=\mathbf 1 \,\right>\nonumber
\end{eqnarray}
 \par\bigskip
$\bullet$ {\bf order 102660} 
\begin{eqnarray}
{A_1(59)}
&\cong& \left< a,b\  |\ aba=bab,\,ab^2a=b^8,\, (a^2b^{-1}a^3b^{-1})^3=\mathbf 1, \, a^{10}=\mathbf 1 \,\right>\quad ^{*\ \dagger}\nonumber
\end{eqnarray}
\par\bigskip
$\bullet$ {\bf order 205320} \quad(Not simple - Schur cover)
\begin{eqnarray}
\overline{A_1(59)}
&\cong& \left< a,b\  |\ aba=bab,\,ab^2a=b^8,\, (a^2b^{-1}a^3b^{-1})^3=\mathbf 1 \,\right>\quad ^{*\ \dagger}\nonumber
\end{eqnarray}
 \par\bigskip

\begin{center} \underline{$\mathbf{n=11}$} \end{center}\par
$\bullet$ {\bf order 660} 
\begin{eqnarray}
{A_1(11)}
&\cong& \left< a,b\  |\ aba=bab,\,ab^2a=b^9,\, (ab^{-1}a^2b^{-1})^3=\mathbf 1 \,\right>\nonumber\\
&\cong& \left< a,b\  |\ aba=bab,\,ab^2a=b^9,\, (a^4b^{-5})^2=\mathbf 1 \,\right>\nonumber\\
&\cong& \left< a,b\  |\ aba=bab,\,ab^2a=b^9,\, (ab^{-2}a^2b^{-2})^2=\mathbf 1 \,\right>\nonumber\\
&\cong& \left< a,b\  |\ aba=bab,\,ab^2a=(a^4b^{-2})^{-3}b^9,\, a^{11}=\mathbf 1 \,\right>\nonumber
\end{eqnarray}
\par\bigskip
$\bullet$ {\bf order 1320} \quad(Not simple - Schur cover)
\begin{eqnarray}
\overline{A_1(11)}
&\cong& \left< a,b\  |\ aba=bab,\,ab^2a=(a^4b^{-2})^{-3}b^9 \,\right>\nonumber\\
&\cong& \left< a,b\  |\ aba=bab,\,ab^2a=(a^2b^{-4})^{3}b^9 \,\right>\nonumber
\end{eqnarray}
 \par\bigskip
$\bullet$ {\bf order 6072} 
\begin{eqnarray}
{A_1(23)}
&\cong& \left< a,b\  |\ aba=bab,\,ab^2a=b^9,\, (a^4b^{-2})^4=\mathbf 1 \,\right>\nonumber\\
&\cong& \left< a,b\  |\ aba=bab,\,ab^2a=b^9,\, (a^2b^{-1}ab^{-3})^3=\mathbf 1 \,\right>\nonumber\\
&\cong& \left< a,b\  |\ aba=bab,\,ab^2a=(ab^{-1}ab^{-3})^3b^9,\, a^{11}=\mathbf 1 \,\right>\nonumber
\end{eqnarray}
\par\bigskip
$\bullet$ {\bf order 12144} \quad(Not simple - Schur cover)
\begin{eqnarray}
\overline{A_1(23)}
&\cong& \left< a,b\  |\ aba=bab,\,ab^2a=(ab^{-1}ab^{-3})^3b^9 \,\right>\nonumber
\end{eqnarray}
\par\bigskip
$\bullet$ {\bf order 32736} 
\begin{eqnarray}
{A_1(32)}
&\cong& \left< a,b\  |\ aba=bab,\,ab^2a=b^9,\, (a^4b^{-2}a^4b^{-1})^2a^{11}=\mathbf 1 \,\right>\quad ^{*\ \dagger}\nonumber
\end{eqnarray}
\par\bigskip
$\bullet$ {\bf order 39732} 
\begin{eqnarray}
{A_1(43)}
&\cong& \left< a,b\  |\ aba=bab,\,ab^2a=b^9,\, (a^2b^{-2})^3a^{11}=\mathbf 1 \,\right>\nonumber\\
&\cong& \left< a,b\  |\ aba=bab,\,ab^2a=b^9,\, ((ab^{-2})^3ab^{-1})^2=\mathbf 1 \,\right>\nonumber\\
&\cong& \left< a,b\  |\ aba=bab,\,ab^2a=b^9,\, ((ab^{-1})^3b^{-1}ab^{-2})^2=\mathbf 1 \,\right>\nonumber
\end{eqnarray}
\par\bigskip
$\bullet$ {\bf order 150348} 
\begin{eqnarray}
{A_1(67)}
&\cong& \left< a,b\  |\ aba=bab,\,ab^2a=b^9,\, ((ab^{-1})^3b^{-1}a^2b^{-2})^2=\mathbf 1 \,\right>\nonumber
\end{eqnarray}
 \par\bigskip
$\bullet$ {\bf order 352440} 
\begin{eqnarray}
{A_1(89)}
&\cong& \left< a,b\  |\ aba=bab,\,ab^2a=b^9,\,  (a^4b^{-3}a^4b^{-1})^2=\mathbf 1 \,\right> \nonumber
\end{eqnarray}
\par\bigskip

\begin{center} \underline{$\mathbf{n=13}$} \end{center}\par
$\bullet$ {\bf order 1092} 
\begin{eqnarray}
{A_1(13)}
&\cong& \left< a,b\  |\ aba=bab,\,ab^2a=b^{11},\,  (a^4b^{-6})^2=\mathbf 1 \,\right>\nonumber\\
&\cong& \left< a,b\  |\ aba=bab,\,ab^2a=b^{11},\,  (ab^{-1}a^3b^{-2}ab^{-1})^2=\mathbf 1 \,\right>\nonumber\\
&\cong& \left< a,b\  |\ aba=bab,\,ab^2a=b^{11},\,  (a^3b^{-1}ab^{-1})^2=\mathbf 1 \,\right>\nonumber\\
&\cong& \left< a,b\  |\ aba=bab,\,ab^2a=b^{11},\,  (a^2b^{-1}ab^{-1}ab^{-1})^2=\mathbf 1 \,\right>\nonumber
\end{eqnarray}
\par\bigskip
$\bullet$ {\bf order 7800} 
\begin{eqnarray}
{A_1(25)}
&\cong& \left< a,b\  |\ aba=bab,\,ab^2a=b^{11},\,  (a^3b^{-4})^2=\mathbf 1 \,\right>\nonumber\\
&\cong& \left< a,b\  |\ aba=bab,\,ab^2a=b^{11},\,  (a^3b^{-1}a^2b^{-2})^2=\mathbf 1 \,\right>\nonumber\\
&\cong& \left< a,b\  |\ aba=bab,\,ab^2a=b^{11},\,  (ab^{-1}a^2b^{-2}a^2b^{-1})^2=\mathbf 1 \,\right>\nonumber
\end{eqnarray}
 \par\bigskip
$\bullet$ {\bf order 9828} 
\begin{eqnarray}
{A_1(27)}
&\cong& \left< a,b\  |\ aba=bab,\,ab^2a=b^{11},\,  (a^2b^{-3})^3=\mathbf 1 \,\right>\nonumber
\end{eqnarray}
 \par\bigskip
$\bullet$ {\bf order 74412} 
\begin{eqnarray}
{A_1(53)}
&\cong& \left< a,b\  |\ aba=bab,\,ab^2a=b^{11},\,  (ab^{-1}(ab^{-2})^2)^2=\mathbf 1 \,\right>\nonumber\\
&\cong& \left< a,b\  |\ aba=bab,\,ab^2a=b^{11},\,  (a^3b^{-1}a^3b^{-2})^2=\mathbf 1 \,\right>\nonumber\\
&\cong& \left< a,b\  |\ aba=bab,\,ab^2a=b^{11},\,  (a^2b^{-1}ab^{-2}ab^{-3})^2=\mathbf 1 \,\right>\nonumber
\end{eqnarray}
 \par\bigskip
$\bullet$ {\bf order 246480} 
\begin{eqnarray}
{A_1(79)}
&\cong& \left< a,b\  |\ aba=bab,\,ab^2a=b^{11},\,  (a^2b^{-1})^4=\mathbf 1 \,\right>\nonumber
\end{eqnarray}
\par\bigskip
$\bullet$ {\bf order 546312} 
\begin{eqnarray}
{A_1(101)}
&\cong& \left< a,b\  |\ aba=bab,\,ab^2a=b^{11},\,  (a^2b^{-1})^4=\mathbf 1 \,\right>\quad ^*\nonumber
\end{eqnarray}
\par\bigskip
$\bullet$ {\bf order 2964780} 
\begin{eqnarray}
{A_1(181)}
&\cong& \left< a,b\  |\ aba=bab,\,ab^2a=b^{11},\,   (a^2b^{-1}a^2b^{-4})^2=\mathbf 1 \,\right>\quad ^*\nonumber
\end{eqnarray}
\par\bigskip

\begin{center} \underline{$\mathbf{n=15}$} \end{center}\par
$\bullet$ {\bf order 4080} 
\begin{eqnarray}
{A_1(16)}
&\cong& \left< a,b\  |\ aba=bab,\,ab^2a=(ab^{-2}ab^{-5})^2b^{13}\,\right>\nonumber
\end{eqnarray}
 \par\bigskip
$\bullet$ {\bf order 12180} 
\begin{eqnarray}
{A_1(29)}
&\cong& \left< a,b\  |\ aba=bab,\,ab^2a=b^{13},\,  (a^3b^{-8})^2=\mathbf 1 \,\right>\nonumber\\
&\cong& \left< a,b\  |\ aba=bab,\,ab^2a=b^{13},\,  (ab^{-1}a^3b^{-1})^2=\mathbf 1 \,\right>\nonumber
\end{eqnarray}
\par\bigskip
$\bullet$ {\bf order 14880} 
\begin{eqnarray}
{A_1(31)}
&\cong& \left< a,b\  |\ aba=bab,\,ab^2a=b^{13},\,  (a^4b^{-6})^2=\mathbf 1 \,\right>\nonumber\\
&\cong& \left< a,b\  |\ aba=bab,\,ab^2a=b^{13},\,  (ab^{-2}ab^{-4})^2=\mathbf 1 \,\right>\nonumber
\end{eqnarray}
 \par\bigskip
$\bullet$ {\bf order 102660} 
\begin{eqnarray}
{A_1(59)}
&\cong& \left< a,b\  |\ aba=bab,\,ab^2a=b^{13},\,  (a^3b^{-4})^2=\mathbf 1 \,\right>\nonumber\\
&\cong& \left< a,b\  |\ aba=bab,\,ab^2a=b^{13},\,  (ab^{-2}a^2b^{-2}ab^{-1})^2=\mathbf 1 \,\right>\nonumber
\end{eqnarray}
 \par\bigskip
$\bullet$ {\bf order 113460} 
\begin{eqnarray}
{A_1(61)}
&\cong& \left< a,b\  |\ aba=bab,\,ab^2a=b^{13},\,  (a^4b^{-5})^2=\mathbf 1 \,\right>\nonumber
\end{eqnarray}
 \par\bigskip

\begin{center} \underline{$\mathbf{n=17}$} \end{center}\par
$\bullet$ {\bf order 2448} 
\begin{eqnarray}
{A_1(17)}
&\cong& \left< a,b\  |\ aba=bab,\,ab^2a=b^{15},\,  (a^3b^{-5})^2=\mathbf 1 \,\right>\nonumber\\
&\cong& \left< a,b\  |\ aba=bab,\,ab^2a=b^{15},\,  (ab^{-1}a^3b^{-4})^2=\mathbf 1 \,\right>\nonumber
\end{eqnarray}
 \par\bigskip
$\bullet$ {\bf order 4080} 
\begin{eqnarray}
{A_1(16)}
&\cong& \left< a,b\  |\ aba=bab,\,ab^2a=(a^3b^{-9})^{2}b^{15} \,\right>\nonumber
\end{eqnarray}
\par\bigskip
$\bullet$ {\bf order 150348} 
\begin{eqnarray}
{A_1(67)}
&\cong& \left< a,b\  |\ aba=bab,\,ab^2a=b^{15},\,  (a^3b^{-4})^2=\mathbf 1 \,\right>\quad ^*\nonumber\\
&\cong& \left< a,b\  |\ aba=bab,\,ab^2a=b^{15},\,  (ab^{-2}ab^{-4})^2=\mathbf 1 \,\right> \nonumber
\end{eqnarray}
\par\bigskip
$\bullet$ {\bf order 515100} 
\begin{eqnarray}
{A_1(101)}
&\cong& \left< a,b\  |\ aba=bab,\,ab^2a=b^{15},\,  (ab^{-2}ab^{-5})^2=\mathbf 1 \,\right>\quad ^*\nonumber\\
&\cong& \left< a,b\  |\ aba=bab,\,ab^2a=b^{15},\,  (a^4b^{-10})^2=\mathbf 1 \,\right>\quad ^*\nonumber
\end{eqnarray}
\par\bigskip

\begin{center} \underline{$\mathbf{n=19}$} \end{center}\par
$\bullet$ {\bf order 3420} 
\begin{eqnarray}
{A_1(19)}
&\cong& \left< a,b\  |\ aba=bab,\,ab^2a=b^{17},\,  (a^4b^{-9})^2=\mathbf 1 \,\right>\nonumber\\
&\cong& \left< a,b\  |\ aba=bab,\,ab^2a=b^{17},\,  (a^3b^{-12})^2=\mathbf 1 \,\right>\nonumber\\
&\cong& \left< a,b\  |\ aba=bab,\,ab^2a=b^{17},\,  (ab^{-1}ab^{-5})^2=\mathbf 1 \,\right>\nonumber
\end{eqnarray}
 \par\bigskip
$\bullet$ {\bf order 25308} 
\begin{eqnarray}
{A_1(37)}
&\cong& \left< a,b\  |\ aba=bab,\,ab^2a=b^{17},\,  (a^3b^{-4})^2=\mathbf 1 \,\right>\nonumber\\
&\cong& \left< a,b\  |\ aba=bab,\,ab^2a=b^{17},\,  (a^4b^{-1}a^2b^{-1})^2=\mathbf 1 \,\right>\nonumber
\end{eqnarray}
 \par\bigskip
$\bullet$ {\bf order 721392} 
\begin{eqnarray}
{A_1(113)}
&\cong& \left< a,b\  |\ aba=bab,\,ab^2a=b^{17},\,  (ab^{-2}ab^{-5})^2=\mathbf 1 \,\right>\quad ^*\nonumber\\
&\cong& \left< a,b\  |\ aba=bab,\,ab^2a=b^{17},\,  (a^4b^{-12})^2=\mathbf 1 \,\right>\quad ^*\nonumber
\end{eqnarray}
\par\bigskip
$\bullet$ {\bf order 1721400} 
\begin{eqnarray}
{A_1(151)}
&\cong& \left< a,b\  |\ aba=bab,\,ab^2a=b^{17},\,  (a^3b^{-3})^2a^{38}=\mathbf 1 \,\right>\quad ^*\nonumber
\end{eqnarray}
\par\bigskip
$\bullet$ {\bf order 6004380} 
\begin{eqnarray}
{A_1(229)}
&\cong& \left< a,b\  |\ aba=bab,\,ab^2a=b^{17},\,  (a^2b^{-3}ab^{-2})^2=\mathbf 1 \,\right>\quad ^*\nonumber\\
&\cong& \left< a,b\  |\ aba=bab,\,ab^2a=b^{17},\,  (a^2b^{-2}ab^{-3})^2=\mathbf 1 \,\right>\quad ^*\nonumber
\end{eqnarray}
\par\bigskip

\begin{center} \underline{$\mathbf{n=21}$} \end{center}\par
$\bullet$ {\bf order 34440} 
\begin{eqnarray}
{A_1(41)}
&\cong& \left< a,b\  |\ aba=bab,\,ab^2a=b^{19},\,  (a^3b^{-7})^2=\mathbf 1 \,\right>\nonumber
\end{eqnarray}
 \par\bigskip
$\bullet$ {\bf order 39732} 
\begin{eqnarray}
{A_1(43)}
&\cong& \left< a,b\  |\ aba=bab,\,ab^2a=b^{19},\,  (a^5b^{-7})^2=\mathbf 1 \,\right>\nonumber\\
&\cong& \left< a,b\  |\ aba=bab,\,ab^2a=b^{19},\,  (a^4b^{-12})^2 =\mathbf 1 \,\right>\nonumber
\end{eqnarray}
 \par\bigskip
$\bullet$ {\bf order 285852} 
\begin{eqnarray}
{A_1(83)}
&\cong& \left< a,b\  |\ aba=bab,\,ab^2a=b^{19},\, (a^3b^{-5})^2 =\mathbf 1 \,\right>\quad ^*\nonumber\\
&\cong& \left< a,b\  |\ aba=bab,\,ab^2a=b^{19},\,  (a^4b^{-6})^2 =\mathbf 1 \,\right>\quad ^*\nonumber
\end{eqnarray}
\par\bigskip
$\bullet$ {\bf order 976500} 
\begin{eqnarray}
{A_1(125)}
&\cong& \left< a,b\  |\ aba=bab,\,ab^2a=b^{19},\, (a^4b^7)^2 =\mathbf 1 \,\right>\quad ^*\nonumber\
\end{eqnarray}
\par\bigskip
$\bullet$ {\bf order 1024128} 
\begin{eqnarray}
{A_1(127)}
&\cong& \left< a,b\  |\ aba=bab,\,ab^2a=b^{19},\, (ab^{-2}ab^{-4})^2 =\mathbf 1 \,\right>\quad ^*\nonumber\
\end{eqnarray}
\par\bigskip

\begin{center} \underline{$\mathbf{n=23}$} \end{center}\par
$\bullet$ {\bf order 6072} 
\begin{eqnarray}
{A_1(23)}
&\cong& \left< a,b\  |\ aba=bab,\,ab^2a=b^{21},\,  (a^3b^{-7})^2=\mathbf 1 \,\right>\nonumber\\
&\cong& \left< a,b\  |\ aba=bab,\,ab^2a=b^{21},\,  (a^4b^{-11})^2=\mathbf 1 \,\right>\nonumber
\end{eqnarray}
 \par\bigskip
$\bullet$ {\bf order 51888} 
\begin{eqnarray}
{A_1(47)}
&\cong& \left< a,b\  |\ aba=bab,\,ab^2a=b^{21},\,  (a^3b^{-12})^2=\mathbf 1 \,\right>\nonumber\\
&\cong& \left< a,b\  |\ aba=bab,\,ab^2a=b^{21},\,  (a^3b^{-13})^2=\mathbf 1 \,\right>\quad ^*\nonumber\\
&\cong& \left< a,b\  |\ aba=bab,\,ab^2a=b^{21},\,  (a^4b^{-12})^2=\mathbf 1 \,\right>\quad ^*\nonumber
\end{eqnarray}
 \par\bigskip
$\bullet$ {\bf order 1285608} 
\begin{eqnarray}
{A_1(137)}
&\cong& \left< a,b\  |\ aba=bab,\,ab^2a=b^{21},\,  (a^3b^{-14})^2=\mathbf 1 \,\right>\quad ^*\nonumber
\end{eqnarray}
\par\bigskip
$\bullet$ {\bf order 1342740} 
\begin{eqnarray}
{A_1(139)}
&\cong& \left< a,b\  |\ aba=bab,\,ab^2a=b^{21},\,  (a^3b^{-16})^2=\mathbf 1 \,\right>\quad ^*\nonumber
\end{eqnarray}
\par\bigskip
$\bullet$ {\bf order 6004380} 
\begin{eqnarray}
{A_1(229)}
&\cong& \left< a,b\  |\ aba=bab,\,ab^2a=b^{21},\,  (a^3b^{-8})^2=\mathbf 1 \,\right>\quad ^*\nonumber
\end{eqnarray}
 \par\bigskip
$\bullet$ {\bf order 10626828} 
\begin{eqnarray}
{A_1(277)}
&\cong& \left< a,b\  |\ aba=bab,\,ab^2a=b^{21},\,  (a^3b^{-9})^2=\mathbf 1 \,\right>\quad ^*\nonumber
\end{eqnarray}
 \par\bigskip

\begin{center} \underline{$\mathbf{n=25}$} \end{center}\par
$\bullet$ {\bf order 60} 
\begin{eqnarray}
{A_1(5)}
&\cong& \left< a,b\  |\ aba=bab,\,ab^2a=b^{23},\,  (ab^{-2})^3=\mathbf 1 \,\right>\nonumber\\
&\cong& \left< a,b\  |\ aba=bab,\,ab^2a=b^{21},\,  (ab^{-3})^2=\mathbf 1 \,\right>\nonumber
\end{eqnarray}
Note: These presentations exhibits $A_1(5)$ as a factor of $D(2,3,25)$. Even so, the order of $a$ and $b$ turns out to be $5$, so it is also a factor of $D(2,3,5)$. \par\bigskip
$\bullet$ {\bf order 58800} 
\begin{eqnarray}
{A_1(49)}
&\cong& \left< a,b\  |\ aba=bab,\,ab^2a=b^{23},\,  (a^3b^{\pm 8})^2 =\mathbf 1 \,\right>\quad ^*\nonumber\\
&\cong& \left< a,b\  |\ aba=bab,\,ab^2a=b^{23},\,  (a^3b^7)^2 =\mathbf 1 \,\right>\quad ^*\nonumber
\end{eqnarray}
\par\bigskip
$\bullet$ {\bf order 515100} 
\begin{eqnarray}
{A_1(101)}
&\cong& \left< a,b\  |\ aba=bab,\,ab^2a=b^{23},\,  (a^3b^{-7})^2 =\mathbf 1 \,\right>\quad ^*\nonumber\\
&\cong& \left< a,b\  |\ aba=bab,\,ab^2a=b^{23},\,  (a^3b^{11})^2 =\mathbf 1 \,\right>\quad ^*\nonumber
\end{eqnarray}
\par\bigskip
$\bullet$ {\bf order 1653900} 
\begin{eqnarray}
{A_1(149)}
&\cong& \left< a,b\  |\ aba=bab,\,ab^2a=b^{23},\,  (a^3b^{-13})^2 =\mathbf 1 \,\right>\quad ^*\nonumber
\end{eqnarray}
 \par\bigskip
$\bullet$ {\bf order 1721400} 
\begin{eqnarray}
{A_1(151)}
&\cong& \left< a,b\  |\ aba=bab,\,ab^2a=b^{23},\,  (a^3b^{-5})^2 =\mathbf 1 \,\right>\quad ^*\nonumber
\end{eqnarray}
\par\bigskip

\begin{center} \underline{$\mathbf{n=27}$} \end{center}\par
$\bullet$ {\bf order 504} 
\begin{eqnarray}
{A_1(8)}
&\cong& \left< a,b\  |\ aba=bab,\,ab^2a=(ab^{-1}ab^{-12})^2b^{25} \,\right>\nonumber
\end{eqnarray}
Note: The order of $a$ and $b$ turns out to be $9$. \par\bigskip
$\bullet$ {\bf order 2448} 
\begin{eqnarray}
{A_1(17)}
&\cong& \left< a,b\  |\ aba=bab,\,ab^2a=b^{25},\, (a^4b^{-14})^2=\mathbf 1 \,\right>\nonumber
\end{eqnarray}
Note: The order of $a$ and $b$ turns out to be $9$. \par\bigskip
$\bullet$ {\bf order 74412} 
\begin{eqnarray}
{A_1(53)}
&\cong& \left< a,b\  |\ aba=bab,\,ab^2a=b^{25},\, (a^6b^{-13})^2a^{54}=\mathbf 1 \,\right>\quad ^*\nonumber
\end{eqnarray}
\par\bigskip
$\bullet$ {\bf order 647460} 
\begin{eqnarray}
{A_1(109)}
&\cong& \left< a,b\  |\ aba=bab,\,ab^2a=b^{25},\, (a^3b^8)^2=\mathbf 1 \,\right>\quad ^*\nonumber\\
&\cong& \left< a,b\  |\ aba=bab,\,ab^2a=b^{25},\, (ab^{-1}ab^{-6})^2=\mathbf 1 \,\right>\quad ^*\nonumber
\end{eqnarray}
 \par\bigskip
$\bullet$ {\bf order 9951120} 
\begin{eqnarray}
{A_1(271)}
&\cong& \left< a,b\  |\ aba=bab,\,ab^2a=b^{25},\, (a^3b^{-7})^2=\mathbf 1 \,\right>\quad ^*\nonumber
\end{eqnarray}
 \par\bigskip

\begin{center} \underline{$\mathbf{n=29}$} \end{center}\par
$\bullet$ {\bf order 12180} 
\begin{eqnarray}
{A_1(29)}
&\cong& \left< a,b\  |\ aba=bab,\,ab^2a=b^{27},\, (a^3b^{-9})^2=\mathbf 1 \,\right>\quad ^*\nonumber
\end{eqnarray}
\par\bigskip
$\bullet$ {\bf order 102660} 
\begin{eqnarray}
{A_1(59)}
&\cong& \left< a,b\  |\ aba=bab,\,ab^2a=b^{27},\, (a^3b^{-7})^2=\mathbf 1 \,\right>\quad ^*\nonumber
\end{eqnarray}
\par\bigskip

\begin{center} \underline{$\mathbf{n=31}$} \end{center}\par
$\bullet$ {\bf order 14880} 
\begin{eqnarray}
{A_1(31)}
&\cong& \left< a,b\  |\ aba=bab,\,ab^2a=b^{29},\, (a^3b^{-20})^2=\mathbf 1 \,\right>\quad ^*\nonumber
\end{eqnarray}
\par\bigskip
$\bullet$ {\bf order 32736} 
\begin{eqnarray}
{A_1(32)}
&\cong& \left< a,b\  |\ aba=bab,\,ab^2a=b^{29},\, (a^3b^{-7})^2=\mathbf 1 \,\right>\quad ^{*\ \dagger}\nonumber
\end{eqnarray}
\par\bigskip

\begin{center} \underline{$\mathbf{n=37}$} \end{center}\par
$\bullet$ {\bf order 25308} 
\begin{eqnarray}
{A_1(37)}
&\cong& \left< a,b\  |\ aba=bab,\,ab^2a=b^{35},\, (a^3b^{-24})^2=\mathbf 1 \,\right>\quad ^*\nonumber
\end{eqnarray}
\par\bigskip

\begin{center} \underline{$\mathbf{n=41}$} \end{center}\par
$\bullet$ {\bf order 34440} 
\begin{eqnarray}
{A_1(41)}
&\cong& \left< a,b\  |\ aba=bab,\,ab^2a=b^{39},\, (a^3b^{-13})^2a^{-82}=\mathbf 1 \,\right>\quad ^*\nonumber
\end{eqnarray}
\par\bigskip

\begin{center} \underline{$\mathbf{n=43}$} \end{center}\par
$\bullet$ {\bf order 39732} 
\begin{eqnarray}
{A_1(43)}
&\cong& \left< a,b\  |\ aba=bab,\,ab^2a=b^{41},\, (a^3b^{15})^2=\mathbf 1 \,\right>\quad ^*\nonumber
\end{eqnarray}
\par\bigskip

\begin{center} \underline{$\mathbf{n=47}$} \end{center}\par
$\bullet$ {\bf order 51888} 
\begin{eqnarray}
{A_1(47)}
&\cong& \left< a,b\  |\ aba=bab,\,ab^2a=b^{45},\, (a^3b^{-15})^2=\mathbf 1 \,\right>\quad ^*\nonumber
\end{eqnarray}
 \par
\end{scriptsize}

\end{document}